\documentclass[11pt]{amsart}
\usepackage[a4paper,margin=3.3cm]{geometry}

\usepackage{amssymb}
\usepackage{amsmath}
\usepackage{amsthm}
\usepackage{array}

\usepackage{enumerate}

\newtheorem{theorem}{Theorem}
\newtheorem{proposition}[theorem]{Proposition}
\newtheorem{lemma}[theorem]{Lemma}
\newtheorem{corollary}[theorem]{Corollary}

\theoremstyle{definition}

\newtheorem{problem}[theorem]{Problem}
\newtheorem*{notation}{Notation}

\theoremstyle{remark}
\newtheorem{remark}[theorem]{Remark}

\newcommand{\f}[1]{\mathbb{F}_{#1}}

\newcommand{\K}{\mathbb{K}}
\newcommand{\F}{\mathbb{F}}

\newcommand{\tracen}[1]{\mathsf{tr}^n_1(#1)}
\newcommand{\tracem}[1]{\mathsf{tr}^m_1(#1)}
\newcommand{\tracend}[2]{\mathsf{tr}^n_{#1}(#2)}
\newcommand{\tracex}[3]{\mathsf{tr}^{#1}_{#2}({#3})}
\newcommand{\tra}[1]{\mathsf{Tr}({#1})}
\newcommand{\Gcd}[2]{\mathsf{gcd}({#1},{#2})}
\newcommand{\Ker}[1]{\mathsf{Ker}({#1})}
\newcommand{\Dim}[1]{\mathsf{dim}\ {#1}}
\newcommand{\Image}[1]{\mathsf{Im}({#1})}
\newcommand{\Ind}[1]{\mathsf{ind}_2({#1})}

\begin{document}


\title{APN trinomials and hexanomials}

\author{Faruk G\"{o}lo\u{g}lu}
\email{farukgologlu@gmail.com}

\date{}

\begin{abstract}
In this paper we give a new family of APN trinomials of the form $X^{2^k+1} + (\tracex{n}{m}{X})^{2^k+1}$ on $\f{2^n}$ where $\Gcd{k}{n}=1$ and $n = 2m = 4t$, and prove its important properties. The family satisfies for all $n = 4t$ an interesting property of the Kim function which is, up to equivalence, the only known APN function equivalent to a permutation on $\f{2^{2m}}$. As another contribution of the paper, we consider a family of hexanomials $g_{C,k}$ which was shown to be differentially $2^{\Gcd{m}{k}}$-uniform by Budaghyan and Carlet (2008) when a quadrinomial $P_{C,k}$ has no roots in a specific subgroup. In this paper, for all $(m,k)$ pairs, we characterize, construct and count all $C \in \f{2^n}$ satisfying the condition. Bracken, Tan and Tan (2014) and Qu, Tan and Li (2014) constructed some elements $C$ satisfying the condition when $m \equiv 2 \textrm{ or } 4 \pmod{6}$ and $m \equiv 0 \pmod{6}$ respectively, both requiring $\Gcd{m}{k} = 1$. Bluher (2013) proved that such $C$ exists if and only if $k \ne m$ without characterizing, constructing or counting those $C$. To prove the results, we effectively use a Trace-$0$/Trace-$1$ (relative to the subfield $\f{2^m}$) decomposition of $\f{2^n}$.
\end{abstract}


\maketitle
\section{Introduction}

A function $f : \f{2^n} =\F \to \F$ is called {\em almost perfect nonlinear (APN)} if the nonzero derivatives $D_A f(X) = f(X)+f(X+A)+f(A)$ are two-to-one maps of $\F$ for $A \in \F^*$. For a long time, the only known APN functions were affinely equivalent to the monomials of Table \ref{APNmono} (or their inverses when they exist or the monomials in the same cyclotomic classes). If $n$ is odd, APN monomials are permutations and if $n = 2m$ is even, they are three-to-one functions. CCZ-equivalence, introduced in \cite{CCZ}, provides the best notion of equivalence regarding differential and linear properties. If $f$ and $f'$ are CCZ-equivalent then $f$ is APN if and only if $f'$ is APN. Many infinite families have been found (see Table \ref{APNmulti}) which are CCZ-inequivalent to monomials. The existence of APN permutations on even $n$ was a big problem, named ``the big APN problem'' by Dillon in \cite{Dillon1}. In the conference $\F_{q^9}$ in 2009, John Dillon presented the first almost perfect nonlinear permutation on even $n = 6$ \cite{Dillon2}. The permutation is CCZ-equivalent to the {\em Kim function} $\kappa(X) = X^3 + X^{10} + A X^{24}$, where $A$ is a generator of $\f{2^6}^*$. As $\kappa$ demonstrates, CCZ equivalence does not preserve being a permutation. Finding an APN permutation on even extension $n > 6$ is regarded as (still) the most important problem of the APN-related research \cite{Dillon2}.

If one wants to find APN permutations on larger even dimensions, a plausible method is to mimic the behaviour of the Kim function $\kappa$, that is, finding ---possibly infinite families of--- APN functions satisfying properties which $\kappa$ satisfies. A property of this kind was stated explicitly in \cite{Dillon2}, which is the existence of some $k$ such that
\[
\kappa(\lambda X) = \lambda^{2^k+1}\kappa(X), \hspace{1.0cm}  \forall \lambda \in \f{2^{m}}
\]
and was called the {\bf subspace property}, and according to Browning {\em et al.} ``explained some of the simplicity'' of the double simplex decomposition of the code corresponding to $\kappa$, i.e., equivalence of $\kappa$ to a permutation. In this paper, we will prove that the function
\[
f_k(X) = X^{2^k+1} + (\tracex{n}{m}{X})^{2^k+1}
\]
is APN if and only if $n = 4t$ and $\Gcd{n}{k} = 1$. This family is, other than the monomials, the first infinite family of APN functions provably satisfying the subspace property. In fact, we are not aware of any APN function (belonging to an infinite family or sporadic) satisfying the subspace property other than the monomials, the Kim function, and the family $f_k$ introduced here. Unfortunately, $f_k$ are not equivalent to permutations on $n = 4,8$ and does not seem to be equivalent to one on $n=12$ (we say ``it does not seem to be equivalent to a permutation'' since checking the existence of CCZ-equivalent permutations requires huge amount of computing and is infeasible on $n=12$; our program was still running at the time of writing). On the other hand, no member of our family is equivalent to a known APN function on $\f{2^{12}}$ and they are inequivalent for different $k$ to each other (considering $1\le k < m$) for $n \le 12$. 

Being {\em crooked} (see the definition in Section \ref{sec_pre}), the derivatives of $f_k$ are hyperplanes. We give the hyperplanes of $\F$ to which the image set of $D_A f_k(X)$ corresponds. Using the {\em hyperplane spectrum} we are able to give the Walsh spectrum of $f_k$ as well as the elements $A$ such that $\tracen{A f_k(X)}$ are bent. Finding the Walsh spectrum of non-monomials is regarded as an interesting problem \cite[Problem 9.2.61]{HFF}. The properties of the hyperplane spectrum imply the Walsh spectrum result and makes it possible to list all the embedded bent functions explicitly.

Budaghyan, Carlet and Leander \cite{B5} found the simply described and beautiful family of APN functions
\[
X^3 + \tracex{n}{1}{X^9}
\]
for any $n$. A technique called ``switching'' was introduced in \cite{EP} which produced new (sporadic) non-quadratic APN functions by adding Boolean functions to known APN functions. Finding infinite families similar to the BCL family is regarded as an interesting research problem. Our new family provides such an example. As a difference, the new family can be seen as adding a {\em vectorial} (i.e., not a single output) Boolean function to the existing Gold family of monomials. Another property of our family, similar to the BCL family, which we believe worth mentioning is that its coefficients are from the simplest possible field, $\f{2}$. 

Another contribution of the paper is related to another family of APN hexanomials $g_{C,k}$ introduced by Budaghyan and Carlet \cite{BC}, and polynomials $P_{C,k}$ of the form $X^{2^k+1} + CX^{2^k}+C^{2^m}X+1$. The polynomials of the form $X^{2^k+1} + SX^{2^k}+TX+U$ have been studied extensively \cite{Abhyankar,Bluher1,HK1,HK2} and appear in many problems related to finite fields: difference sets \cite{DillonGeo,DD}, cross-correlation of $m$-sequences \cite{DFHR,HKN}, error-correcting-codes \cite{BraHel} and quite recently, the discrete logarithm problem on finite fields \cite{GGMZa,GGMZb}. Budaghyan and Carlet \cite{BC} showed that $g_{C,k}$ are differentially $2^{\Gcd{m}{k}}$-uniform (APN if $\Gcd{m}{k}=1$), provided that $P_{C,k}$ has no roots in the cyclic subgroup $\mathcal{P}_{2^m-1}$ of $\f{2^n}^*$ of order $2^{m}+1$. Bracken, Tan and Tan constructed \cite{BTT} some elements $C$ when $m \equiv 2 \textrm{ or } 4 \pmod{6}$  such that $P_{C_k}$ has no roots in $\mathcal{P}_{2^m-1}$ (in the $\Gcd{m}{k}=1$ case). Later, Qu, Tan and Li \cite{Qu} constructed some elements when $m \equiv 0 \pmod{6}$. Bluher characterized those $(m,k)$ pairs for which such a $P_{C,k}$ exists for any $\Gcd{m}{k}$ without giving a characterization, construction or counting method for $C$. In this paper, we completely characterize all $(C,m,k)$ triples, that is, for any $(m,k)$ pair, we describe all elements $C$ such that $P_{C,k}$ has no roots in $\mathcal{P}_{2^m-1}$ with no $\Gcd{m}{k}$ restriction. We are also able to count all such $C$ for any $(m,k)$ pair and construct them efficiently.

For all contributions of the paper we use a method based on a {\bf Trace-$0$/Trace-$1$} (relative to the subfield $\f{2^m}$) decomposition. We switch between the Trace-$0$/Trace-$1$ and the better-known {\bf polar-coordinate} (subfield $\f{2^m}$-sub\-group $\mathcal{P}_{2^m-1}$) decompositions by the two maps $\phi,\psi$ we introduce in the next section. The polar-coordinate decomposition has been widely used unlike the Trace-$0$/Trace-$1$ representation. Our paper can be seen as an introduction to a basic method which does not require complicated character sums and is effective especially when dealing with quadratic exponents $2^k+1$ and $2^m(2^k+1)$ together. We first give the basics of our method in Section \ref{sec_pre}. Then in Section \ref{sec_apn}, we will prove almost perfect nonlinearity and other properties of the new APN family. In Section \ref{sec_hex}, we will prove our results related to BC hexanomials.

\begin{table}[!h]
\noindent\begin{center} 
{\footnotesize
\begin{tabular}{|c|c|c|c|c|} 
\hline 
 \textbf{Family} & \textbf{Monomial} & \textbf{Conditions} & \textbf{Proved in}\\ 
\hline 
{Gold} & \footnotesize{$X^{2^i+1}$} & \footnotesize{$\Gcd{i}{n}=1$} & \footnotesize{\cite{Gold}}\\ 
\hline 
{Kasami} & \footnotesize{$X^{2^{2i}-2^i+1}$} & \footnotesize{$\Gcd{i}{n}=1$}  & \footnotesize{\cite{kasami-71}}\\ 
\hline 
{Welch}  & \footnotesize{$X^{2^t +3}$ }&  \footnotesize{$n=2t+1$}  &\footnotesize{\cite{Dobb1}}\\ 
\hline 
{Niho}  &\footnotesize{$X^{2^t+2^\frac{t}{2}-1}$, $t$ even} & \footnotesize{$n=2t+1$} & \footnotesize{\cite{Dobb2}}\\ 
 & \footnotesize{$X^{2^t+2^\frac{3t+1}{2}-1}$, $t$ odd} &  & \\ 
\hline 
{Inverse} &\footnotesize{$X^{2^{2t}-1}$}& \footnotesize{$n=2t+1$}& \footnotesize{\cite{Nyb94}}\\ 
\hline 
{Dobbertin}  & \footnotesize{$X^{2^{4t}+2^{3t}+2^{2t}+2^{t}-1}$} & \footnotesize{$n=5t$} & \footnotesize{\cite{dobbertin-power5}}\\ 
\hline 
\end{tabular} 
}
\end{center}
\caption{Known infinite families of APN monomials on $\f{2^n}$} \label{APNmono}
\end{table}

\begin{table}[!h]
\noindent\begin{center} 
{\footnotesize
\begin{tabular}{|c|m{4.0cm}|m{4.0cm}|c|} 
\hline 
 \textbf{\#} & \textbf{Polynomial} & \textbf{Conditions} & \textbf{Proved in}\\ 
\hline 
 B.1 
 &
 $X^{2^s+1}+ A^{2^t-1} X^{2^{it}+2^{rt+s}}$ 
 & $n =3t$, $\Gcd{t}{3}=\Gcd{s}{3t}=1$, $t \geq 3$, $i \equiv st \pmod{3}$, $r= 3-i$, $A \in \F$ is primitive 
 & \cite{B1} \\
\hline 
 B.2 
 &
 $X^{2^s+1}+ A^{2^t-1} X^{2^{it}+2^{rt+s}}$
 & $n =4t$, $\Gcd{t}{2}=\Gcd{s}{2t}=1$, $t \geq 3$, $i \equiv st \pmod{4}$, $r= 4-i$, $A \in \F$ is primitive 
 & \cite{B2} \\
\hline 
 B.3 
 &
	$A X^{2^{s}+1}+A^{2^m}X^{2^{m+s}+2^m}+B X^{2^m+1}+\sum_{i=1}^{m-1} c_i X^{2^{m+i}+2^i}$
	& $n = 2m$, $m$ odd, $c_i \in \f{2^m}$, $\Gcd{s}{m} = 1$, $s$ is odd, $A,B \in \F$ primitive 
	& \cite{B3} \\
\hline 
 B.4 
 &
 $AX^{2^{n-t}+2^{t+s}}+A^{2^{t}}X^{2^{s}+1}+b X^{2^{t+s}+2^{s}}$
 & $n = 3t$, $\Gcd{s}{3t}=1$, $\Gcd{3}{t}=1$, $3|(t+s)$, $A \in \F$ primitive, $b \in \f{2^t}$
 & \cite{B3} \\
\hline
 B.5 
 &
 $ A^{2^t}X^{2^{n-t}+2^{t+s}}+A X^{2^{s}+1}+b X^{2^{n-t}+1}$
 & $n=3t$, $\Gcd{s}{3t}=\Gcd{3}{t}=1$, $3|(t+s)$, $A \in \F$ primitive, $b \in \f{2^t}$ 
 & \cite{B4}\\
\hline 
 B.6 
 &
 $ A^{2^t}X^{2^{n-t}+2^{t+s}}+A X^{2^{s}+1}+b X^{2^{n-t}+1}+c A^{{2^t}+1}X^{2^{t+s}+2^s}$
 & $n=3t$, $\Gcd{s}{3t}=\Gcd{3}{t}=1$, $3 |(t+s)$, $A \in \F$ primitive, $b,c\in \f{2^t}$, $bc \ne 1$
 & \cite{B4}\\
\hline
 B.7 
 &
 $X^{2^{2k}+2^k} + B X^{q+1} + C X^{q(2^{2k}+2^k)}$
 & $n=2m$, $m$ odd, $C$ is a $(q-1)$st power but not a $(q-1)(2^i+1)$st power, $CB^q+B \ne 0$
 & \cite{BC}\\
\hline 
 B.8 
 &
 $X (X^{2^k} + X^q + C X^{2^k q}) + X^{2^k}(C^q X^q + A X^{2^k q}) + X^{(2^k+1)q}$ 
 & $n=2m$, $\Gcd{n}{k} = 1$, $C$ satisfies Theorem \ref{thm_hex}, $A \in \F \setminus \f{2^m}$
 & \cite{BC} \\
\hline 
 B.9 
 &
 $X^3 + \tracex{n}{1}{X^9}$ 
 & 
 & \cite{B5}\\
\hline 
 B.10 
 &
 $X^{2^k+1} + \tracex{n}{m}{X}^{2^k+1}$ 
 & $n = 2m = 4t$, $\Gcd{n}{k} = 1$
 & here\\
\hline 
 B.11
 &
 Bivariate construction Theorem 1 of \cite{CarletBV}
 & $n = 2m$ 
 & \cite{CarletBV} \\
\hline
 B.12
 &
 Bivariate construction Theorem 9 of \cite{Zhou}
 & $n = 4m$ 
 & \cite{Zhou} \\
\hline

\end{tabular} 
}
\end{center}
\caption{Known infinite families of APN multinomials on $\f{2^n}$}\label{APNmulti}
\end{table}

\section{Preliminaries}\label{sec_pre}

\begin{notation} 
Throughout the paper the following notation is used.
\[
\begin{array}{ll}
q = 2^m & \\
n = 2m & \\ 
\K = \f{q}, &  x,y,a,b \in \K \\
\F = \f{q^2}, &  X,Y,A,B \in \F \\
& \\
\tra{X} = X + X^q & \\
\tracex{e}{d}{X} = \sum_{i=0}^{\frac{e}{d}-1} X^{2^{id}} & \\
H_\beta = \{ X \in \F \ : \ \tracen{\beta X} = 0 \} & \\
& \\
\mathcal{T}_1 = \{X \in \f{q^2} \ : \ \tra{X} = 1 \}, & g,h \in \mathcal{T}_1 \\
\mathcal{P}_{q-1} = \{X^{q-1} \ : \ X \in \F^*\}, & u,v \in \mathcal{P}_{q-1} \\
\phi : \F \to \mathcal{P}_{q-1}, & \phi : X \mapsto X^{q-1} \\
\psi : \F \setminus \K \to \mathcal{T}_1, & \psi : X \mapsto \frac{X}{\tra{X}}\\
& \\
\Gamma_k : \K \to \K, & \Gamma_k : x \mapsto x^{2^k+1} + x
\end{array}
\]
\end{notation}

\subsection{Definitions and basics}

In this paper we will use a {\bf Trace-$0$/Trace-$1$} decomposition of $\F^*$. Any $X \in \F^*$ can be written uniquely as $X = x g$ where $x \in \K^*$ and $g \in \mathcal{T}_1 \cup \{1\}$. If $xg = yh$ then $\tra{xg} = \tra{yh} = 0$ implies $g=h=1$ and therefore $x=y$. If $\tra{xg} = \tra{yh} \ne 0$, then $\tra{xg} = \tra{yh} = x = y$ and therefore $h=g$. There is another decomposition of $\F^*$ which is well-known and usually called the {\bf polar-coordinate} decomposition. Any $X \in \F^*$ can be written as $X = x u$ where $x \in \K^*$ and $u \in \mathcal{P}_{q-1}$. If $xu = yv$ then $(xu)^{q-1} = (yv)^{q-1}$ means $u^2=v^2$ and therefore $x = y$.

For $g \in \mathcal{T}_1$, we have $g^q = g+1$. For any fixed $g \in \mathcal{T}_1$, we can write any $h \in \mathcal{T}_1$ as $h=g+a$ for a unique $a \in \K$. Similarly, for any fixed $g \in \mathcal{T}_1$, any $X \in \F$ can be written as $X = a g + b$ where $a,b \in \K$.

The map $\phi$ maps $\mathcal{T}_1$ to $\mathcal{P}_{q-1} \setminus \{ 1 \}$ bijectively. Indeed $g^{q-1} = h^{q-1}$ implies $\frac{g^q}{g} = \frac{h^q}{h}$ which implies $\frac{h+1}{h}=\frac{g+1}{g}$ and $g=h$. Similarly, $\psi$ maps $\mathcal{P}_{q-1} \setminus \{ 1 \}$ onto $\mathcal{T}_1$, since $\frac{u}{\tra{u}} = \frac{v}{\tra{v}}$, taking $(q-1)$st powers, gives $u^2=v^2$. When considering the maps between $\mathcal{T}_1$ and $\mathcal{P}_{q-1} \setminus \{ 1 \}$ we will use the inverse maps $\phi^{-1}$ and $\psi^{-1}$. Note that 
\[
\begin{array}{ll}
\phi^{-1}: \mathcal{P}_{q-1} \setminus \{1\} \to \mathcal{T}_1, & \phi^{-1}: u \mapsto (\frac{u}{\tra{u}})^{\frac{q}{2}},\\
\psi^{-1}: \mathcal{T}_1 \to \mathcal{P}_{q-1} \setminus \{1\}, & \psi^{-1}: g \mapsto g^{\frac{(q-1)q}{2}}.
\end{array}
\]

A function $\F \to \F$ is called a vectorial Boolean function. A cryptographically important measure for vectorial Boolean functions is the {\em differential uniformity}. The differential uniformity of $f$ is defined as
\[
\delta_f = \max_{ A \ne 0,B \in \F} \# \left\{ X \in \F \ : \ f(X) + f(X+A) = B \right\}.
\]
Optimal functions with respect to differential uniformity in characteristic $2$ are called {\em almost perfect nonlinear (APN)}, and satisfy $\delta_f = 2$. 
This means, if $\delta_f = 2$ then the (normalized) derivatives of $f$
\[
D_A f (X) = f(X) + f(X+A) + f(A) + f(0), \ A \in \F^* 
\]
are two-to-one maps. The {\em crooked functions} are vectorial Boolean functions whose (normalized) derivatives are two-to-one maps, image of which are hyperplanes. Therefore, the crooked functions are APN. All quadratic APN functions are crooked since their derivatives are linearized polynomials. The existence of a non-quadratic crooked function is an interesting open problem. Such a function cannot be a monomial \cite{Kyu1} or a binomial \cite{Kyu2} if exists at all. For a crooked function $f$, the hyperplane spectrum $\mathcal{H}_f$ is defined by the multiset  
\[
\mathcal{H}_f = \left\{* \ \beta \in \F^* \ : \   \Image{D_A f} = H_\beta \ *\right\}.
\]
For instance for the Gold monomials $f_{\textsf{Gold}}(X) = X^{2^k+1}$ on $\f{2^n}$ with $\Gcd{n}{k}=1$, it is common knowledge and trivial to show that the image sets of the derivatives $D_A f_{\textsf{Gold}}(X) = X^{2^k+1} + (X+A)^{2^k+1} + A^{2^k+1}$ are of the form
\[
\left\{ A^{2^k} X + A X^{2^k} \ : \ X \in \F \right\} = \left\{ A^{2^k+1} (X + X^{2^k}) \ : \ X \in \F \right\}
\]
and therefore $\mathcal{H}_{f_{\textsf{Gold}}} = \{* \ \beta^{2^k+1} \ : \ \beta \in \F^* \ *\}$, which is equivalent to the set of nonzero cubes each with multiplicity three if $n$ is even and the set of all nonzero elements each with multiplicity one if $n$ is odd. It is also well known that \cite[Lemma 5]{CanChaDec} when $n$ is odd, the image sets of the derivatives of a crooked function correspond to all hyperplanes with multiplicity one.

The {\em Walsh transform} of $f$
\[
\widehat{f}(A,B) = \sum_{X \in \F} \chi \left(A f(X) + B x\right)
\]
where $\chi(\cdot) = (-1)^{\tracen{\cdot}}$, is another important tool. For the known APN functions, the Walsh spectrum $\{ \widehat{f}(A,B) \ : \ (0,0) \ne (A,B) \in \F \times \F \}$ consists of values which are usually ``low'' in absolute value. If $n$ is odd, the quadratic APN functions all have the same {\em optimal} spectrum $\{ 0, \pm 2^{\frac{n+1}{2}} \}$ (these functions are optimal in the sense that the maximal absolute value of the Walsh spectrum is the smallest possible and called {\em almost bent (AB)}). If $n = 2m$ is even, almost all of the known quadratic APN functions  have the spectrum $\{ 0, \pm 2^m, \pm 2^{m+1}\}$ with a single exception on $\f{2^6}$ which has the spectrum $\{ 0, \pm 2^m, \pm 2^{m+1}, \pm 2^{m+2} \}$ (see \cite{Dillon1}). Generalizing this sporadic example to an infinite class is a very interesting open problem.

For a given function $f$, the set of zeroes of the Walsh transform $\{ (X,Y) \ : \ \widehat{f}(X,Y) = 0 \} \cup \{ (0,0) \}$ has an interesting connection to being a permuta\-tion/CCZ-equivalent to a permutation. A function $f$ is a permutation if and only if the {\em lines} $(0,Y)$ and $(X,0)$ are zeroes of the Walsh transform where $X,Y \in \F$. CCZ-equivalence maps the subspaces of the Walsh zeroes of $f$ to the subspaces of the Walsh zeroes of $f'$ and therefore $f$ is CCZ-equivalent to a permutation if and only if the zeroes of the Walsh transform of $f$ contains two subspaces of dimension $n$ intersecting only trivially \cite{Dillon2}. We say that $f$ satisfies the {\em subspace property} \cite{Dillon2} if
\begin{equation}\label{SP}
	f(aX) = a^{2^k+1}f(X), \hspace{1.0cm} \forall a \in \K.
\end{equation}
for some integer $k$. The Kim function $\kappa(X) = X^3 + X^{10} + A X^{24}$, where $A$ is a generator of $\f{2^6}^*$ satisfies the subspace property. A result of this is the set of Walsh zeroes of $\kappa$ has more structure with respect to subspaces $u\K$ where $u \in \mathcal{P}_7$, that is to say the zeroes of the Walsh spectrum contains the subspaces $\{ (u_1 x, v_1 y) \ : \ x,y \in \K \}, \{ (u_2 x, v_2 y) \ : \ x,y \in \K \}$ for some $u_1,u_2,v_1,v_2 \in \mathcal{P}_7$. According to \cite{Dillon2}, this explains the simplicity of the double simplex decomposition of the code corresponding to $\kappa$. See \cite{Dillon2} for details and some other interesting properties implied by the subspace property.

We refer the reader to excellent surveys for more on APN, AB, crooked and other nonlinear functions \cite[Chapter ``Vectorial Boolean functions for cryptography'' by Carlet]{CRAMA}, \cite[Chapter ``PN and APN functions'' by Charpin]{HFF}, \cite[Chapter ``Special mappings of finite fields'' by Kyureghyan]{RICAM}.

\subsection{An introductory lemma}
The following lemma contains some basics of our method and will be used extensively throughout the paper. Recall than $n=2m$ and $q = 2^m$. We define $\Ind{k}$ to be the largest positive integer $e$ such that $2^e | k$.

\begin{lemma}\label{tracelemma}
Let $g \in \mathcal{T}_1$. 
\begin{enumerate}[(i)]
\item $\tra{g^{2^k+1}} = g^{2^k}+g+1$.
\item $g^{2^k+1} = \tra{g^{2^k+1}} g + \tra{g^3} + 1$.

\item For integers $d,e > 0$, we have 
\begin{itemize}
\item $\Gcd{2^d-1}{2^e+1} > 1$ if and only if $\Ind{d} > \Ind{e}$,
\item $\Gcd{2^d+1}{2^e+1} > 1$ if and only if $\Ind{d} = \Ind{e}$,
\item $1 \in \{\Gcd{2^d-1}{2^e+1}, \Gcd{2^d+1}{2^e+1} \}$.
\end{itemize}

\item Let $\Gcd{m}{k} = d$. Then
\[
\tracend{d}{g^{2^k+1}} = \left\{ \begin{array}{rl}
1 & \textrm{if } \Gcd{2^k+1}{q+1} = 1,\\
0 & \textrm{otherwise}.
\end{array}\right.
\]
Also
\[
\tracen{g^{2^k+1}} = \left\{ \begin{array}{rl}
1 & \textrm{if $m+k$ is odd},\\
0 & \textrm{if $m+k$ is even}.
\end{array}\right.
\]

\item Let $Z_{k,\epsilon} = \{ g \in \mathcal{T}_1 \ : \ g^{2^k}+g = \epsilon\}$ for $\epsilon \in \f{2}$. We have 

\begin{itemize}
\item $\phi(Z_{k,0}) \cup \{1\}$ is the set of $\Gcd{2^k-1}{q+1}$st roots of unity in $\mathcal{P}_{q-1}$,
\item $\phi(Z_{k,1}) \cup \{1\}$ is the set of $\Gcd{2^k+1}{q+1}$st roots of unity in $\mathcal{P}_{q-1}$,
\item $\#Z_{k,0} = {\Gcd{2^k-1}{q+1}} -1$,
\item $\#Z_{k,1} = {\Gcd{2^k+1}{q+1}} -1$.
\end{itemize}

\item Let $\Gcd{k}{n} = 1$. Then $\tra{g^{2^k+1}} = 0$ if and only if $m$ is odd and $g \in \f{4} \setminus \f{2}$.
\end{enumerate}
\end{lemma}

\begin{proof} Let $g \in \mathcal{T}_1$.
\begin{enumerate}[(i)]
\item Using $g^q = g + 1$, we get $\tra{g^{2^k+1}} =  g^{2^k+1}+(g+1)^{2^k+1} = g^{2^k}+g+1$. 

\item Obvious after applying Lemma \ref{tracelemma} (i) to $2^k+1$ and $3$.  

\item We have 
\[
\Gcd{2^d-1}{2^e+1} = \frac{2^{\Gcd{2e}{d}}-1}{2^{\Gcd{e}{d}}-1}
\]
and 
\[
\Gcd{2^d+1}{2^e+1} = \frac{(2^{\Gcd{e}{d}}-1)(2^{\Gcd{2e}{2d}}-1)}{(2^{\Gcd{2e}{d}}-1)(2^{\Gcd{e}{2d}}-1)}.
\]
The former is strictly greater than $1$ if and only if $\Ind{d} > \Ind{e}$, and the latter is strictly greater than $1$ if and only $\Ind{d} = \Ind{e}$. Clearly, both cannot happen simultaneously.

\item Let $m=m'd$ and $k=k'd$. We have
\begin{align*}
\tracex{m}{d}{g^{2^k}+g+1} & = \sum_{i=0}^{m'-1} g^{2^{di}} + \sum_{i=k'}^{m'+k'-1} g^{2^{di}}  + \sum_{i=0}^{m'-1} 1\\
& = \sum_{i=0}^{m'-1} g^{2^{di}} + \sum_{i=k'}^{m'-1} g^{2^{di}} + \sum_{i=0}^{k'-1} (g+1)^{2^{di}} +\sum_{i=0}^{m'-1} 1\\
& = \sum_{i=0}^{m'-1} g^{2^{di}} + \sum_{i=0}^{m'-1} g^{2^{di}} + \sum_{i=0}^{k'-1} 1 +\sum_{i=0}^{m'-1} 1\\
& = k'+m'.
\end{align*}

We have $k'+m'$ even if and only if $\Ind{k} = \Ind{m}$ if and only if $\Gcd{2^k+1}{q+1} > 1$. For the second claim, observe that if $m+k$ is odd then $\Gcd{2^k+1}{q+1} = 1$ and $d$ is odd. We have $\tracex{d}{1}{\tracex{n}{d}{g^{2^k+1}}} = 1.$ If $m+k$ is even then $\Gcd{2^k+1}{q+1} > 1$ and $\tracex{d}{1}{\tracex{n}{d}{g^{2^k+1}}} = 0.$


\item Since $\phi$ is a bijection from $\mathcal{T}_1$ to $\mathcal{P}_{q-1} \setminus \{ 1 \}$, we have
\begin{align}
\{\phi(g) \ : \ g \in \mathcal{T}_1 \ \textrm{ s.t. } \ g^{2^i-1} = 1\} & = \{ u \in \mathcal{P}_{q-1} \setminus \{ 1 \} \ : \ \phi^{-1}(u)^{2^i-1} = 1 \} \nonumber \\
		& = \{ u \in \mathcal{P}_{q-1}\setminus \{ 1 \}  \ : \ \left(\frac{u^{2^i-1}}{\tra{u}^{2^i-1}}\right)^{\frac{q}{2}} = 1 \}\nonumber \\
		& = \{ u \in \mathcal{P}_{q-1}\setminus \{ 1 \}  \ : \ {u^{2^i-1}} = 1 \} \nonumber \\ 
		& = \{ u \in \mathcal{P}_{q-1}\setminus \{ 1 \}  \ : \ {u^{\Gcd{2^i-1}{q+1}}} = 1 \} \label{rootseq}
\end{align}
The penultimate line follows from the observation that if $u^{2^i-1} = 1$ then $u \in \f{2^i}$ which forces $\tra{u} \in \f{2^i}$, conversely $u^{2^i-1} \ne 1$ implies $u^{2^i-1} \not\in \K$ and $\frac{u^{2^i-1}}{\tra{u}^{2^i-1}} \ne 1$.

We have $g^{2^k} + g = 0 \iff g^{2^k-1} = 1$ and the first result follows if we set $i = k$ in Eq. \eqref{rootseq}. The cardinality follows from the fact that $\mathcal{P}_{q-1}$ is a cyclic group of order $q+1$. Similarly, $g^{2^k} + g = 1 \iff g^{2^{2k}-1} = 1$ and $g^{2^k-1} \ne 1$. Note that in this case $2k \mid 2m$ therefore $2^k-1 \mid q-1$ and $\Gcd{2^{2k}-1}{q+1} = \Gcd{2^k+1}{q+1}$. Since $k\mid m$, $g^{2^k-1} \ne 1$ since $\mathcal{T}_1 \cap \f{2^k} = \emptyset$. Setting $i = 2k$ in Eq. \eqref{rootseq}, we prove the second result and the corresponding cardinality. 

\item By Lemma \ref{tracelemma} (i), we have $\tra{g^{2^k+1}}=0$ if and only if $g + g^{2^k} = 1$. Since $\Gcd{n}{k} = 1$, the equation $X + X^{2^k} = 1$  has  two solutions $\omega,\omega^2 \in \f{4} \setminus \f{2}$. Since $\omega,\omega^2 \in \mathcal{T}_1$ if and only if $m$ is odd, we are done.



\end{enumerate}
\end{proof}

\subsection{A technical lemma}
The following lemma is rather technical and each part of it will be used in just one instance in Theorem \ref{thm_count} and Theorem \ref{thm_hyp} respectively.

\begin{lemma}\label{techlemma}
Let 
\[
C_{k,0} = \left\{ \frac{g^{2^k+1}+g}{g^{2^k}+g} \ : \ g \in \mathcal{T}_1 \setminus Z_{k,0} \right\},
\]
and
\[
D_{k,1} = \left\{ \frac{g^{2^k+1}+1}{g^{2^k}+g+1} \ : \ g \in \mathcal{T}_1 \setminus Z_{k,1} \right\}.
\]
We have 
\begin{enumerate}[(i)]
\item
\begin{itemize}
\item $C_{k,0} \subseteq \mathcal{T}_1$,
\item $\#C_{k,0} = \frac{q+1}{\Gcd{2^k-1}{q+1}}-1$, 
\item $C_{k,0} = \mathcal{T}_1 \iff \Gcd{2^k-1}{q+1} = 1$.
\end{itemize}
\item 
\begin{itemize}
\item $D_{k,1} \subseteq \mathcal{T}_1$,
\item $D_{k,1} = \mathcal{T}_1 \iff \Gcd{2^k+1}{q+1}=1$,
\end{itemize}

\end{enumerate}
\end{lemma}
\begin{proof}
\begin{enumerate}[(i)]
\item Using Lemma \ref{tracelemma} (i) one derives that the set $C_{k,0}$ consists of elements of the form $\frac{X}{\tra{X}}$ and therefore $C_{k,0} \subseteq \mathcal{T}_1$. We will count the number of elements under the $\phi$ map. The numerators satisfy $g^{2^k+1} + g= g(g+1)^{2^k} = g^{2^k q+1}$. Define $d = \Gcd{2^k q + 1}{q+1}$. Note that $d = \Gcd{2^k-1}{q+1}$. Now if $g \in \mathcal{T}_1 \setminus Z_{k,0}$ then $\phi\left(\frac{g^{2^k q+1}}{g^{2^k}+g}\right) = \phi(g^{2^k q+1})$ and
\begin{align*}
\left\{ \phi\left(g^{2^k q+1}\right) \ : \ g \in \mathcal{T}_1 \setminus Z_{k,0} \right\} & = \left\{ \phi(g)^{2^k q+1} \ : \ g \in \mathcal{T}_1 \setminus Z_{k,0} \right\} \\
 & = \left\{ u^{2^k q+1} \ : \ u \in \mathcal{P}_{q-1} \setminus \left( \phi(Z_{k,0}) \cup \{1\} \right) \right\}\\
 & = \left\{ u^d \ : \ u \in \mathcal{P}_{q-1} \right\} \setminus \{1\},
\end{align*}
since $u^d$ and $u^{2^k q +1}$ are $d$-to-$1$ maps on $\mathcal{P}_{q-1}$ with the identical image sets which map only the $d$ elements of the set $\phi(Z_{k,0}) \cup \{1\}$ of $d$-th roots in $\mathcal{P}_{q-1}$ to $1$ by Lemma \ref{tracelemma} (v). 
\item As in Part (i), $D_{k,1}$ consists of elements of the form $\frac{X}{\tra{X}}$ and therefore $D_{k,1} \subseteq \mathcal{T}_1$. If $\Gcd{2^k+1}{q+1} > 1$ then $\#Z_{k,1} > 0$ and $D_{k,1} \ne \mathcal{T}_1$. Now assume $\Gcd{2^k+1}{q+1} = 1$. If $D_{k,1} \ne \mathcal{T}_1$, then there exists $a \in \K^*$ such that
\[
\frac{g^{2^k+1}+1}{g^{2^k}+g+1} = \frac{(g+a)^{2^k+1}+1}{(g+a)^{2^k}+(g+a)+1}.
\]
Recall that denominators are nonzero since $Z_{k,1} = \emptyset$. 
We then have, 
\begin{align*}
\frac{g^{2^k+1}+1}{g^{2^k}+g+1} & = \frac{(g^{2^k+1}+1)+a^{2^k+1}+a^{2^k}g+a g^{2^k}}{(g^{2^k}+g+1)+a^{2^k}+a} 
\end{align*}
which implies 
\[
(g^{2^k}+g+1)(a^{2^k+1}+a^{2^k}g+a g^{2^k}) = (g^{2^k+1}+1)(a^{2^k}+a)
\]
and in turn
\[
a^{2^k} S + a S^{2^k} + a^{2^k+1} T = 0
\]
where $S = \tra{g^3}$ and $T = \tra{g^{2^k+1}}$. By Lemma \ref{tracelemma} (iv), $T \ne 0$. If $S = 0$ then $a = 0$ and we are done. If $S \ne 0$ then put $a = bS$. Then
\[
S^{2^k+1}(b^{2^k}+b+b^{2^k+1} T) = 0
\]
and 
\begin{equation}\label{contra}
T = \frac{1}{b} + \frac{1}{b^{2^k}}.
\end{equation}
Let $d = \Gcd{m}{k}$. By Lemma \ref{tracelemma} (iv), $\tracex{m}{d}{T} = 1$ but $\tracex{m}{d}{\frac{1}{b}+ (\frac{1}{b})^{2^k}} = \tracex{m}{d}{\frac{1}{b}+ (\frac{1}{b})^{2^d}} = 0$, which contradicts \eqref{contra}.

\end{enumerate}
\end{proof}

\section{The new APN trinomial family}\label{sec_apn}

In this section, we will prove that $f_k$ is APN on $\f{2^n}$ if and only if $\Gcd{k}{n} = 1$ and $n = 2m = 4t$ where $t > 0$. Then we will prove its important properties such as the hyperplane and Walsh spectra, description of the embedded bent functions, subspace property and inequivalence to other families.

\subsection{Proof of almost perfect nonlinearity}

\begin{theorem}\label{thm_apn}
Let $f_k(X) = X^{2^k+1} + (\tra{X})^{2^k+1}$. Then $f_k$ is APN if and only if $m$ is even and $\Gcd{k}{n} = 1$.
\end{theorem}

\begin{proof}
We have $f_k(X) = X^{(2^k + 1)q} + X^{2^k q+1} + X^{2^k+q}$. Let 
\[ 
D_A f_k(X) = f_k(X)+f_k(X+A)+f_k(A). 
\]
Clearly, $f_k$ is APN if and only if $D_A f_k(X)$ is two-to-one for all $A \in \F^*$. Let $A = ag$, where $a \in \K^*$ , and $g \in \mathcal{T}_1 \cup \{1\}$.
We have
\begin{align*}
D_A f_k(X) & = f_k(X)+f_k(X+A)+f_k(A)\\
			 & = A^q X^{2^k q} + A^{2^k q} X^q + A X^{2^k q} + A^{2^k q} X + A^q X^{2^k} + A^{2^k} X^q\\
			 & = A^{2^k q} X + A^q X^{2^k} + \tra{A^{2^k}} X^q + \tra{A} X^{2^k q}
\end{align*}
 If $g = 1$ then $A = a$ and $\tra{a} = 0$. Therefore 
\begin{equation}
D_a f_k(X) = a^{2^k} X + a X^{2^k} \label{hypeq1}
\end{equation}
is two-to-one if and only if $\Gcd{k}{n} = 1$.

Therefore we focus only on $\Gcd{k}{n}=1$ case. Now assume $g \in \mathcal{T}_1$ and $h = g^q=g+1$. We get 
\[
D_A f_k(X) = (ah)^{2^k} X + (ah) X^{2^k} + a^{2^k} X^q + a X^{2^k q}.
\]
We apply the change of variable $X = aY$ and get 
\begin{align}
D_A f_k(Y) & = a^{2^k+1} \left(h^{2^k} Y + h Y^{2^k} + Y^q + Y^{2^k q}\right) \nonumber \\
				 & = a^{2^k+1} \left(g Y^{2^k} + g^{2^k} Y + \tra{Y + Y^{2^k}}\right). \label{hypeq2} 
\end{align}
Therefore, we will only be interested in two-to-oneness of 
\[
L_g(X) = g X^{2^k} + g^{2^k} X + \tra{X + X^{2^k}}. 
\]

Let for any fixed $g$, $X = xg+y$, where $x,y \in \K$.
\begin{align}
L_g(X) & = g X^{2^k} + g^{2^k} X + \tra{X + X^{2^k}} \nonumber\\
			 & = g (xg+y)^{2^k} + g^{2^k} (xg+y) + \tra{(xg+y) + (xg+y)^{2^k}} \nonumber\\
			 & = x^{2^k} g^{2^k+1} + y^{2^k} g + x g^{2^k+1} +y g^{2^k} + x + x^{2^k} \nonumber\\
			 & = (x+x^{2^k}) g^{2^k+1} + y^{2^k} g + y g^{2^k} + x + x^{2^k}. \label{eqforl}
\end{align}

Consider the case $m$ is odd. Let $\omega \in\f{4}\setminus \f{2}$. We have $L_\omega (X) = (\omega^{2^k+1}+1) (x+x^{2^k}) + (y \omega^{2^k}+ y^{2^k} \omega )$. Therefore $L_\omega(x\omega) = 0$ for all $x \in \K$ (i.e., when $y=0$), since $k$ is odd and $\omega^{2^k+1} = \omega^3 = 1$. Hence $D_{\omega}f_k$ are not 2-to-1.

Now, we need to focus only on $m$ even and $\Gcd{n}{k} = 1$ case. We have, from Lemma \ref{tracelemma} (vi), $\tra{g^{2^k+1}} \ne 0$. 

Assume now $x,y \in \f{2}$. By Lemma \ref{tracelemma} (iv) and $m$ is even, $\tracem{g^2+g} \ne 0$, therefore for any $g$, $L_g(xg+y) = 0$ if and only if $x \in \f{2}$ and $y = 0$ that is $X \in \{0,g\}$. Therefore $\Dim{\Ker{L_g}} \ge 1$. To prove two-to-oneness we have to show that these are the only solutions of $L_g(X)$ for all $g$.

If $L_g(X) = 0$ then $\tra{g X^{2^k} + g^{2^k} X} = 0$ and therefore (by \eqref{eqforl})
\begin{equation}\label{treq}
(x + x^{2^k}) \tra{g^{2^k+1}}+ (y + y^{2^k}) = 0. 
\end{equation}



We will assume $L_g(X) = 0$ and $x,y \not\in \f{2}$. This implies
\[
			(x + x^{2^k}) g^{2^k+1} + y^{2^k} g + y g^{2^k} + x + x^{2^k} = 0.
\]
Since we have assumed $L_g(X) = 0$, we obviously have $\tra{L_g(X)} = 0$ and we can employ \eqref{treq}, which implies
\[
x + x^{2^k} =  \frac{y + y^{2^k}}{\tra{g^{2^k+1}}}.
\]
Setting $T = \tra{g^{2^k+1}}$, we get
\[
L_g(X) = \frac{y + y^{2^k}}{T} g^{2^k+1} + y^{2^k} g + y g^{2^k} + \frac{y + y^{2^k}}{T} = 0.
\]
Since $g^{2^k+1} = T g + S + 1$ where $S = \tra{g^3}$ by Lemma \ref{tracelemma} (ii), we have
\begin{align*}
L_g(X) & = (y + y^{2^k}) \left(g + \frac{S+1}{T}\right) + y^{2^k} g + y g^{2^k} + \frac{y + y^{2^k}}{T}\\
			 & = (y + y^{2^k}) \left(g + \frac{S+1}{T}\right) + (y + y^{2^k}) g + y (g + g^{2^k}) + \frac{y + y^{2^k}}{T}\\
			 & = g (y + y^{2^k} + y + y^{2^k}) + \frac{(y+y^{2^k})(S+1)}{T} + \frac{y (T^2+T)}{T} + \frac{y + y^{2^k}}{T}\\
			 & = \frac{(y^{2^k}+y)S + y (T^2 + T)}{T} = 0
\end{align*}
from which we get
\[
S y^{2^k}  + S^{2^k} y = 0,
\]
since $T + T^2 = S+S^{2^k}$ by Lemma \ref{tracelemma} (i). Since $S y^{2^k} + S^{2^k} y$ is two-to-one on $\K$, we have $y = S$ (we are under assumption $y \ne 0$). Therefore 
\[
T = \frac{y^{2^k}+y}{x^{2^k}+x} = \frac{S^{2^k}+S}{x^{2^k}+x} = \frac{T^{2}+T}{x^{2^k}+x}
\]
implies $T + 1= x^{2^k} + x$. Since $\tracem{T+1} = 1$ by Lemma \ref{tracelemma} (iv), the equation cannot be satisfied for any $x \in \K \setminus \f{2}$. This proves the result.  

\end{proof}

\subsection{Hyperplane and Walsh spectra of the family}

In the last section we showed that the derivative $D_A f_k (X) = f_k(X) + f_k(X+A) + f_k(A)$ of $f_k$ is an hyperplane of the form
\[
\left\{ a^{2^k+1} \left( g^{2^k} X + g X^{2^k} + \tra{X + X^{2^k}}\right) \ : \ X \in \F \right\}
\]
with $A = ag$, where $g \in \mathcal{T}_1$ and $a \in \K^*$ or
\[
\left\{ a^{2^k+1} \left(X + X^{2^k}\right) \ : \ X \in \F \right\}
\]
for  $A = a \in \K^*$. The next theorem provides the hyperplane spectrum of the new APN family $f_k$. For a linearized polynomial
\[
L(X) = \sum_{i = 0}^{n-1} a_i X^{2^i},
\]
the adjoint of $L$ is defined as
\[
L^*(X) = \sum_{i=0}^{n-1} a_i^{2^{n-i}} X^{2^{n-i}}.
\]

\begin{theorem}\label{thm_hyp}
Let $A = ag$ where $a \in \K^*$ and $g \in \mathcal{T}_1$. Then the derivatives $D_A f_k$ of $f_k$ satisfy
\[
\Image{D_A f_k} = H_{\beta_A}
\]
where 
\[
\beta_A = \frac{1}{a^{2^k+1}} \frac{\tra{g^{2^k+1}}}{\tra{g^3}^{2^k+1}} \left(g + 1 + \frac{\tra{g^3}}{\tra{g^{2^k+1}}}\right).
\]
The hyperplane spectrum $\mathcal{H}_{f_k}$ of $f_k$ has constant multiplicity $3$.
\end{theorem}
\begin{proof}
Let $A = ag$ where $a \in \K^*$ and $g \in \mathcal{T}_1 \cup \{1\}$. If $A \in \K^*$ then $\beta_a = \frac{1}{a^{2^k+1}}$ by the Eq. \eqref{hypeq1}. Otherwise, $g \ne 1$ and we have $\beta_A =  \frac{1}{a^{2^k+1}} \beta_g$ by the Eq. \eqref{hypeq2}. We will find for each $g \in \mathcal{T}_1$ the unique $\beta_g \in \F^*$ which satisfies $\tracen{\beta_g L_g(X)} = 0, \ \forall X \in \F$ where $L_g(X) = g^{2^k} X + g X^{2^k} + \tra{X + X^{2^k}}$. This is equivalent to finding the unique nonzero solution of the adjoint map
\[
L_g^*(X) = g^{2^k} X + g^{2^{n-k}} X^{2^{n-k}} + \tra{X^{2^{n-k}} + X}
\]
which is guaranteed to have a unique nonzero solution because the adjoint preserves the dimension of the kernel of a linear map and  satisfies $\tracen{\beta_g L_g(X)} =  \tracen{L_g^*(\beta_g) X}$. 
Instead of the actual adjoint map, we will use
\[
M(X) = (L_g^*(X))^{2^k} = g^{2^{2k}} X^{2^k} + g X + \tra{X + X^{2^k}}
\]
which has the same kernel as $L_g^*(X)$.

Now, write $X = xg + y$, where $x,y \in \K$. We get
\begin{align*}
M(X) & = g^{2^{2k}} (xg+y)^{2^k} + g (xg+y) + x + x^{2^k} \\
				 & = x^{2^k} g^{2^{2k}+2^k} + y^{2^k} g^{2^{2k}} + x g^2 + yg + x + x^{2^k}.
\end{align*}
Now
\begin{align*}
M(X) = & x^{2^k} (T g + S+1)^{2^k} + y^{2^k} g + y^{2^k}(T^{2^k}+T)\\
	& + x g + x (S+1) + yg + x + x^{2^k}\\
= & x^{2^k} (T^{2^k} (g + T + 1) + S^{2^k}+1) + g (y^{2^k} + y + x) \\
  &	+ y^{2^k}(T^{2^k}+T) + x (S+1) + x + x^{2^k}\\
= & g (y^{2^k} + y + T^{2^k} x^{2^k} + x) \\
  & + y^{2^k}(T^{2^k}+T) + x (S+1) + x + x^{2^k} + x^{2^k} (T^{2^k} (T + 1) + S^{2^k}+1) \\
\end{align*}
where $T = \tra{g^{2^k+1}}$ and $S = \tra{g^3}$. 
If $M(X) = 0$ then $\tra{M(X)} = 0$ and therefore 
\begin{equation}
T^{2^k} x^{2^k} + x + y^{2^k} + y = 0. \label{tracecond}
\end{equation}
Since we want to get the zeroes of $M(X) = 0$, we proceed
\begin{align*}
0 & = y^{2^k}(T^{2^k}+T) + x (S+1) + x + x^{2^k} + x^{2^k} (T^{2^k} (T + 1) + S^{2^k}+1) \\
& = y^{2^k}(T^{2^k}+T) + x (S+T)+ T (T^{2^k}x^{2^k} + x) +x^{2^k}(S+T)^{2^k}\\
& = y^{2^k}(T^{2^k}+T) + x (S+T)+ T (y^{2^k}+y) +x^{2^k}(S+T)^{2^k}\\
& = (yT)^{2^k} + yT + x (S+T) + (x(S+T))^{2^k}
\end{align*}
where we used \eqref{tracecond} in line 3. This implies 
\begin{align}
y = & x \frac{S+T}{T}, \textrm{ or} \label{pos1}\\
y = & x\frac{S+T}{T} + \frac{1}{T} \label{pos2}
\end{align}
since $T$ is nonzero by Lemma \ref{tracelemma} (vi).
Employing \eqref{pos1} in \eqref{tracecond} we get

\begin{align*}
\left(x \frac{S+T}{T}\right)^{2^k} + x\frac{S+T}{T} + T^{2^k}x^{2^k} + x  & = 0 \\
x^{2^k} \frac{T^{2^k+1}+(T+S)^{2^k}}{T^{2^k}} + x \frac{S}{T} & = 0\\
x^{2^k} \frac{(T^2+T+S)^{2^k}}{T^{2^k}} + x \frac{S}{T} & = 0\\
x^{2^k} \frac{S^{2^{2k}}}{T^{2^k}} + x \frac{S}{T} & = 0
\end{align*}
since $T^2 + T = S^{2^k}+S$. This implies (since $S$ is nonzero by Lemma \ref{tracelemma} (vi))
\[
x^{2^k-1} = \frac{T^{2^k-1}}{S^{2^{2k}-1}} 
\]
which has the nonzero solution $x = \frac{T}{S^{2^k+1}}$. Since $M(X)$ must have unique nonzero solution, we do not need to check \eqref{pos2} which is guaranteed not to provide another solution. Therefore 
\[
\beta_g  = \frac{T}{S^{2^k+1}} \left(g + 1 + \frac{S}{T}\right),
\]
and therefore $\beta_A = \frac{1}{a^{2^k+1}}\beta_g$. Therefore the hyperplane spectrum of $f$ is
\[
\mathcal{H}_f = \left\{* \ a^3 \ : \ a \in \K^* \ *\right\} \cup \left\{* \ a^3 \beta_g  : \ a \in \K^*, h \in \mathcal{T}_1 \ *\right\}
\]
since $\Gcd{2^k+1}{q^2-1} = 3$. Now the only thing left to show is that the multiplicity of each element in $\mathcal{H}_f$ is $3$.
To that end note that
\[
g+1+\frac{S_g}{T_g}=\frac{g^{2^k+1}+1}{g^{2^k}+g+1}.
\]
By Lemma \ref{techlemma} (ii), $\{ g+1+\frac{S_g}{T_g} : g \in \mathcal{T}_1\} = \mathcal{T}_1$. Therefore $\beta_{ag} = \beta_{bh}$ if and only if $g = h$ and $a^3 = b^3$.  
\end{proof}
\begin{remark}
If $k = 1$ then $\beta_A = \frac{g}{a^3 \tra{g^3}^2}$ becomes very simple.
\end{remark}

The theorem has the following corollary, proof of which is standard. The reader is encouraged to read the detailed explanations in \cite[Section 3.1.3 of the Chapter ``Vectorial Boolean functions for cryptography'']{CRAMA}.

\begin{corollary}\label{corr} We have
\begin{enumerate}[(i)]
\item The Walsh spectrum $\mathcal{W}_{f_k}$ of $f_k$ satisfies $\mathcal{W}_{f_k} = \{ 0, \pm 2^{m}, \pm 2^{m+1} \}$, 
\item If $A \in \F^*$ and $A^{-1} \not\in \mathcal{H}_f$, then the binomial (monomial if $A \in \K^*$) Boolean function $\tracen{A X^{2^k+1} + (A^q+A) X^{q2^k+1}}$ is bent. The number of such bent functions is $2 \frac{q^2-1}{3}$.
\end{enumerate}
\end{corollary}
\begin{proof}
\begin{enumerate}[(i)]
\item 
If $A = 0$, then $\widehat{f_k}(0,B) = 0$ if $B \ne 0$. Let $A \in \F^*$.
\begin{align}
(\widehat{f_k}(A,B))^2 & = \sum_{X \in \F} \chi (A f_k(X) + B X) \sum_{Y \in \F} \chi (A(f_k(X)+f_k(X+Y)+f_k(Y))) \nonumber \\ 
		& = q^2 \left(1 + \sum_{\substack{ X \in \F^* \\ \beta_X = A^{-1}}} \chi \left(A f_k(X) + B X \right)\right) \label{benteq}
\end{align}
Since by Theorem \ref{thm_hyp}, $\beta_X = A^{-1}$ happens exactly $3$ times when $A^{-1} \in \mathcal{H}_f$ and $\beta_X \ne A^{-1}$ otherwise. Therefore  $\mathcal{W}_{f_k} \subseteq \{ 0, \pm 2^{m}, \pm 2^{m+1} \}$. Since $\mathcal{W}_{f_k} \subseteq \{ 0, \pm 2^{m}, \pm 2^{m+1} \}$ implies  $\mathcal{W}_{f_k} = \{ 0, \pm 2^{m}, \pm 2^{m+1} \}$ (cf. \cite[Sec. 3.1.3 of Chapter ``Vectorial Boolean functions for cryptography'']{CRAMA}), we prove the first part.
\item When $\beta_X \ne A^{-1}$, the right hand side of Eq. \eqref{benteq} is always $q^2$, implying that $\tracen{A f_k(X)}$ is bent. 
\end{enumerate}
\end{proof}
\begin{remark}
The bent functions of Corollary \ref{corr} themselves are not interesting as they are all quadratic. We believe on the other hand, the description of $A$ for which $\tracen{A f(X)}$ is bent is interesting, since it means that $(A,X)$ cannot be a Walsh zero of $f$ for any $X \in \F$. Recall that, to show that an APN function $f$ is CCZ-equivalent to a permutation one has to find two subspaces of dimension $n$ in the set of Walsh zeroes of $f$ intersecting only trivially.
\end{remark}
\subsection{The subspace property}

Note that for $X = xg$ where $x \in \K^*$ and $g \in \mathcal{T}_1 \cup \{1\}$, $f_k(X) = x^{2^k+1} f_k(g)$. Since $f_k(1) \ne 0$ and $f_k(g) = g^{2^k+1} + g^{2^k} + g \ne 0$ because $\tra{g^{2^k+1}+g^{2^k}+g} = \tra{g^{2^k+1}} \ne 0$ by Lemma \ref{tracelemma} (vi). We have 

\begin{theorem}
The functions $f_k$ satisfy the subspace property , i.e.,
\[
f_k(aX) = a^{2^k+1} f_k(X), \hspace{1.0cm} \forall a \in \K.
\]
\end{theorem}

\subsection{Inequivalence}

The only families that completely covers our parameter condition $n = 4t$ are the monomials of Table \ref{APNmono}, and B.8, B.9, B.11 and B.12 of Table \ref{APNmulti}. The functions $f_k$ are not CCZ-equivalent to B.8, B.9, B.11, B.12 or any monomial other than the Gold family on $\f{2^{8}}$. They distinguish themselves from the Gold monomials on $\f{2^{12}}$. In fact, on $\f{2^{12}}$ our functions are not equivalent to any member of a family we have tried until now (different choices of $A,B,a,b$ etc. in Table \ref{APNmulti}) for each possible parameter set ($s,t,k$ etc. in Table \ref{APNmulti}) for each possible family. Note that checking CCZ-equivalence to every member of every family for every parameter set is infeasible (one equivalence check requires an hour on a standard computer using Magma) and also different parameters and coefficients give (experimentally speaking) generally the same function up to equivalence within a family. The equivalence check was done by checking CCZ-equivalence in Magma \cite{Magma} since a theoretical method does not exist. Moreover, $f_k$ is not equivalent to $f_l$ for $1 \le k \ne l < m$ on $\f{2^8}$ and $\f{2^{12}}$.  

We have checked whether our families are CCZ-equivalent to permutations using a C program. They are not equivalent to permutations on $\f{2^4}$ or $\f{2^8}$. The case $\f{2^{12}}$ (and larger fields) was inconclusive at the time of writing because of the immense computation required. 

\section{Budaghyan-Carlet hexanomial APNs}\label{sec_hex}

Budaghyan and Carlet in \cite[Theorem 2]{BC} proved the following theorem.

\begin{theorem}\cite{BC}
Let $C \in \F$ and $A \in \F \setminus \K$. If 
\[
P_{C,k}(X) = X^{2^k+1}+C X^{2^k} + C^q X + 1 = 0
\]
has no solutions $X \in \mathcal{P}_{q-1}$, then the polynomial
\[
g_{C,k}(X) = X (X^{2^k} + X^q + C X^{2^k q}) + X^{2^k}(C^q X^q + A X^{2^k q}) + X^{(2^k+1)q}
\]
is differentially $2^{\Gcd{k}{m}}$-uniform on $\F$. Thus, $g_{C,k}$ is APN if and only if $\Gcd{k}{m} = 1$.
\end{theorem}

Let for the sequel, $\Gamma_k : \K \to \K$ with $\Gamma_k : x \mapsto x^{2^k+1} + x$. Bracken, Tan and Tan gave \cite[Theorem 2.1]{BTT} a characterization of elements $a \in \K \setminus \Image{\Gamma_k}$ when $\Gcd{k}{m} = 1$. Using this characterization, they constructed some elements $C$  for which $P_{C,k}(X)$ has no solutions when $m \equiv 2 \textrm{ or } 4 \pmod{6}$. Later, Qu, Tan and Li \cite{Qu} did the same thing for $m \equiv 0 \pmod{6}$. Bluher \cite{Bluher2} characterized all admissible $(m,k)$ pairs (if $1 \le k < n$, the condition is $k\ne m$) by giving an existence proof (i.e., without constructing or characterizing such $C$). In this paper, we characterize and construct all elements $C$ for which $P_{C,k}(X)$ has no solutions in $\mathcal{P}_{q-1}$ for all $(m,k)$ pairs and count them. Our theorem makes listing all such elements possible very efficiently. 
Helleseth and Kholosha \cite{HK1,HK2} gave a detailed analysis of the zeroes of $x^{2^k+1}+x+a$ and in particular, the exact cardinality of $\Image{\Gamma_k}$. 
\begin{proposition}{\cite[Proposition 2]{HK2}}\label{prop_hk}
\[
\#\Image{\Gamma_k} = q - \left\{\begin{array}{rl}
\frac{(q+1)2^{\Gcd{k}{m}-1}}{2^{\Gcd{k}{m}}+1} & \textrm{if $\frac{m}{\Gcd{k}{m}}$ is odd,}\\
\frac{(q-1)2^{\Gcd{k}{m}-1}}{2^{\Gcd{k}{m}}+1} & \textrm{if $\frac{m}{\Gcd{k}{m}}$ is even.}
\end{array}\right.
\]
\end{proposition}

The following theorem gives a complete characterization (of $C,k,m$) on when $P_{C,k}(X)$ has no solutions $X \in \mathcal{P}_{q-1}$.

\begin{theorem}\label{thm_hex}
Let $C \in \F$ and $1 \le k < n$. The polynomial 
\[
P_{C,k}(X) = X^{2^k+1}+C X^{2^k} + C^q X + 1
\]
has a solution $X \in \mathcal{P}_{q-1}$ if and only if each of the three following conditions holds
\begin{itemize}
\item $k \ne m$, 
\item $C \not\in \K$, and 
\item 
\[
\frac{\tra{h^3}+1+\frac{1}{b}}{\tra{h^{2^k+1}}^{2^{n-k}+1}} \not\in \Image{\Gamma_k}
\]
where $C^q + 1 = bh$ with $b \in \K^*$ and $h \in \mathcal{T}_1 \setminus Z_{k,1}$. 
\end{itemize}
\end{theorem}

\begin{proof}
Note that $P_{C,k}(1) = 0$ if and only if $C \in \K$. So we assume $C \not\in \K$ and restrict our attention to the set $\mathcal{P}_{q-1} \setminus \{1\}$. Therefore $P_{C,k}(u)$ has a solution $u \in \mathcal{P}_{q-1}$ if and only if $Q_C^{'}(g) = \phi(g)^{2^k+1}+C \phi(g)^{2^k} + C^q \phi(g) + 1$ has a solution $g \in \mathcal{T}_1$. 
\begin{align*}
Q_C^{'}(g) & = g^{(q-1)(2^k+1)}+C g^{(q-1)2^k} + C^q g^{q-1} + 1\\
			     & = \frac{g^{q(2^k+1)}}{g^{2^k+1}} + \frac{C g^{q 2^k}}{g^{2^k}} + \frac{C^q g^{q}}{g} + 1\\
			 & = \frac{g^{q(2^k+1)} + C g^{q 2^k + 1} + C^q g^{q+2^k} + g^{2^k+1}}{g^{2^k+1}}.
\end{align*}
Now, $Q_C^{'}(g)$ has a solution $g \in \mathcal{T}_1$ if and only if $Q_C(g) = g^{q(2^k+1)} + C g^{q 2^k + 1} + C^q g^{q+2^k} + g^{2^k+1}$  has a solution $g \in \mathcal{T}_1$. And, 
\begin{align*}
Q_C(g) & = g^{q(2^k+1)} + C g^{q 2^k + 1} + C^q g^{q+2^k} + g^{2^k+1}\\
			 & = (g+1)^{2^k+1} + C (g+1)^{2^k} g +C^q (g+1) g^{2^k} + g^{2^k+1}\\
			 & = (C+C^q) g^{2^k+1} + (C^q+1) g^{2^k} + (C+1) g + 1.
\end{align*}
Now set $B = C^q+1 = bh$ for some $b \in \K$ and $h \in \mathcal{T}_1$ (note that $h \ne 1$ as $B \in \K$ if and only if $C \in \K$). We now have $Q_C(g) \ne 0 $ if and only if $b  g^{2^k+1} + bh g^{2^k} + (bh)^q g \ne 1$ for all $g \in \mathcal{T}_1$, or for all $x \in \K$
\begin{align}
(h+x)^{2^k+1} + h (h+x)^{2^k} + (h+1) (h+x) & \ne \frac{1}{b} \nonumber\\ 
x^{2^k+1}+x (h^{2^k}+h+1)	& \ne \frac{1}{b} + h^2+h	\label{lasteq} 
\end{align}
since every $g = h+x$ for some $x \in \K$. Now assume $h^{2^k}+h=1$, i.e. $h \in Z_{k,1}$. By Lemma \ref{tracelemma} (v), $\Gcd{2^k+1}{q+1}>1$ and by Lemma \ref{tracelemma} (iii), $\Gcd{2^k+1}{q-1} = 1$. Then \eqref{lasteq} becomes $x^{2^k+1} \ne \frac{1}{b} + h^2+h$, which is impossible to satisfy for any $b$ since $x^{2^k+1}$ is a permutation of $\K$. Note that by Lemma \ref{tracelemma} (v), this happens for $\Gcd{2^k+1}{q+1}-1$ elements $h \in \mathcal{T}_1$ which means all $h \in \mathcal{T}_1$ if and only if $k=m$.

If $h \not\in Z_{k,1}$, then by applying the change of variable $x = y(h^{2^k}+h+1)^{2^{n-k}}$ in \eqref{lasteq}, we get 
\[
y^{2^k+1} + y \ne \frac{h^2+h+\frac{1}{b}}{(h^{2^k} + h + 1)^{2^{n-k}(2^k+1)}}
\]
or equivalently
\begin{equation}
\frac{1}{b} \ne A_h (y^{2^k+1} + y) + B_h. \label{eqwithinv}
\end{equation}
where $A_h = (h^{2^k} + h + 1)^{2^{n-k}(2^k+1)}$ and $B_h = h^2+h$. 
For a fixed $h \not\in Z_{k,1}$, the size of the image set of the right hand side is $\#\Image{\Gamma_k}$. Therefore \eqref{lasteq} is always satisfied for all $h \not\in Z_{k,1}$, for some $b \in \K^*$ since $\#\Image{\Gamma_k} < q-1$ if $k \ne m$ (see Proposition \ref{prop_hk}). 

\end{proof}

Next, we will count the number of $C \in \F$ for which the polynomial $P_{C,k}(X)$ has no solutions $X \in \mathcal{P}_{q-1}$. The cardinality $\#\Image{\Gamma_k}$ is given explicitly in Proposition \ref{prop_hk}.

\begin{theorem}\label{thm_count}
The number of elements $C \in \F$ for which the polynomial $P_{C,k}(X)$ has no solutions $X \in \mathcal{P}_{q-1}$ is 
\begin{align*}
N_{m,k} =	& (q-\Gcd{2^k+1}{q+1}+1)(q-\#\Image{\Gamma_k}-1) \\
		  & + \frac{q+1}{\Gcd{2^k-1}{q+1}} -\Gcd{2^k+1}{q+1}.
\end{align*}
In particular, if $\Gcd{k}{m} = 1$ (i.e., $g_{C,k}$ is APN), then 
\[
N_{m,k} = \left\{\begin{array}{rl}
(q-2) \frac{q+1}{3} & \textrm{if $m$ is odd},\\
q \frac{q-1}{3} & \textrm{if $m$ is even}.
\end{array}\right.
\]
\end{theorem}
\begin{proof}
Recall that $P_{C,k}(1)=0$ if and only if $C \in \K$, therefore we will assume $C = bh$ where $b \in \K^*$ and $h \in \mathcal{T}_1$. Recall the inequality \eqref{eqwithinv}
\[
\frac{1}{b} \ne A_h (x^{2^k+1} + x) + B_h,
\]
and note that the cardinality of the image set of the right hand side is the same as $\#\Image{\Gamma_k}$ if and only if $A_h \ne 0$. Also, if $h^{2^k}+h=1$ (i.e., $A_h = 0$), then there exists no $b$ satisfying \eqref{eqwithinv} (see the proof of Theorem \ref{thm_hex}). For all $h \in \mathcal{T}_1$ such that $h^{2^k}+h \ne 1$ (i.e. $h \not\in Z_{k,1}$ of Lemma \ref{tracelemma} (v), where the number of such $h$ is given explicitly), there are at least $q-\#\Image{\Gamma_k}-1$ elements $b \in \K^*$ satisfying \eqref{eqwithinv} for some $y \in \K$, since $b$ covers all nonzero elements. Therefore the number of $C = bh$ satisfying \eqref{eqwithinv} for some $x \in \K$ is
\[
N_{m,k} = \sum_{h \in \mathcal{T}_1 \setminus Z_{k,1}} \#n_h
\]
where $n_h = \{ b \in \K^* \ : \ A_h (x^{2^k+1} + x) + B_h \ne \frac{1}{b}, \ \forall x \in \K \}$. Note that  
\[
n_h = \left\{ \begin{array}{ll}
q - \Image{\Gamma_k}  & \textrm {if } A_h (x^{2^k+1} + x) + B_h = 0 \textrm{ for some } x \in \K,\\
q - \Image{\Gamma_k} - 1 & \textrm {if } A_h (x^{2^k+1} + x) + B_h \ne 0 \textrm{ for all } x \in \K.
\end{array}
\right.
\]
Therefore we need to count for how many $h \in \mathcal{T}_1 \setminus Z_{k,1}$ there exists $x \in \K$ such that $A_h (x^{2^k+1} + x) + B_h =0,$ or equivalently
\begin{equation}
x^{2^k+1}+x (h^{2^k}+h+1)+h^2+h = 0. \label{zeroeq}
\end{equation}

We have $h^{2^k+1}+ h (h^{2^k}+h+1) + h^2+h = 0$ (note that this does not provide a solution as $h\not\in\K$). And $x = h+g \in \K$ is a solution of the Eq. \eqref{zeroeq} if and only if
\begin{align*}
0 & = (h+g)^{2^k+1}+ (h+g) (h^{2^k}+h+1) + h^2+h  \\
	& = g^{2^k+1} + g h^{2^k} + g^{2^k}h + g (h^{2^k}+h+1) + h^{2^k+1}+ h (h^{2^k}+h+1) + h^2+h\\
	& = g^{2^k+1} + g h^{2^k} + g^{2^k}h + g (h^{2^k}+h+1) \\
	& = g^{2^k+1} + g + h ( g^{2^k} + g).
\end{align*}
Since $g^{2^k}+g = 0$ does not lead to a solution, the last line gives $h = \frac{g^{2^k+1}+g}{g^{2^k} + g}$. The number of $h \in \mathcal{T}_1$ which satisfies the condition $h^{2^k}+h \ne 1$ and for which such a $g \in \mathcal{T}_1$ exists is $\frac{q+1}{\Gcd{2^k-1}{q+1}} -\Gcd{2^k+1}{q+1}$. Indeed, the number of $h \in \mathcal{T}_1$ such that $h = \frac{g^{2^k+1}+g}{g^{2^k} + g}$ for some $g \in \mathcal{T}_1$ is $\frac{q+1}{\Gcd{2^k-1}{q+1}} - 1$ by Lemma \ref{techlemma} (i). From this number, we need to subtract $\Gcd{2^k+1}{q+1} - 1$ elements $h$ which satisfy $h^{2^k}+h = 1$. Since, if $h^{2^k}+h = 1$ then by Lemma \ref{tracelemma} (v) and (iii), $\Gcd{2^k-1}{q+1} = 1$ and by Lemma \ref{techlemma} (i), the set $C_{k,0} = \mathcal{T}_1$, which means that there exists a $g$ satisfying $h = \frac{g^{2^k+1}+g}{g^{2^k} + g}$ (if there are no $h$ such that $h^{2^k}+h = 1$ then $\Gcd{2^k+1}{q+1} - 1 = 0$ and we had subtracted nothing).

The number of $C \in \F$ such that $P_{C,k}(X)$ has no solutions $X \in \mathcal{P}_{q-1}$ is therefore
\begin{align*}
N_{m,k} =	& (q-\Gcd{2^k+1}{q+1}+1)(q-\#\Image{\Gamma_k}-1) \\
		  & + \frac{q+1}{\Gcd{2^k-1}{q+1}} -\Gcd{2^k+1}{q+1}.
\end{align*}
Employing the cardinality of $\Image{\Gamma_k}$ given in Proposition \ref{prop_hk}, we get the number of Budaghyan-Carlet APN functions for a given $(m,k)$ pair.
\end{proof}

\begin{remark}
We would like to remark that the method to list all $C \in \F$ such that $P_{C,k}(X)$ has no solutions $X \in \mathcal{P}_{q-1}$ our Theorem \ref{thm_hex} provides is very efficient: First, determine $\K \setminus \Image{\Gamma_k}$ either directly or using the methods of Bracken, Tan and Tan \cite[Theorem 2.1]{BTT} or Helleseth and Kholosha \cite{HK1,HK2} (requires time polynomial in $q$ which is the square-root of the field size) and then for each $h \in \mathcal{T}_1$ solve for $b$ (requires time polynomial in $q$). Therefore one needs $\mathcal{O}(q^2)$ field operations to list $\mathcal{O}(q^2)$ elements. 

\begin{remark}
Budaghyan and Carlet \cite{BC} notes that the (experimental) number of $C$ such that $g_{C,1}$ is irreducible is roughly $\frac{3}{10} q^2$, whereas our theorem shows the number of $C$ such that $g_{C,k}$ does not have solutions in $\mathcal{P}_{q-1}$ is roughly $\frac{1}{3} q^2$. This means that most of such $C$ leads to irreducible polynomials.
\end{remark}
\end{remark}

\begin{table}[!h]
\noindent\begin{center} 
{\footnotesize
\begin{tabular}{|c|c|c|c|c|c|c|c|} 
\hline 
 $m \backslash k$ & $1$ & $2$ & $3$ & $4$ & $5$ & $6$ & $7$ \\ 
\hline 
$1$ & $0$ &&&&&&\\
\hline
$2$ & $4$ & $0$ &&&&&\\
\hline 
$3$ & $18$ & $18$ & $0$ &&&&\\
\hline 
$4$ & $80$ & $96$ & $80$ & $0$ &&&\\
\hline
$5$ & $330$ & $330$ & $330$ & $330$ & $0$ &&\\
\hline
$6$ & $1344$ & $1560$ & $1792$ & $1612$ & $1344$ & $0$ &\\
\hline
$7$ & $5418$ & $5418$ & $5418$ & $5418$ & $5418$ & $5418$ & $0$ \\
\hline
\end{tabular} 
}
\end{center}
\caption{Some values of $N_{m,k}$}
\end{table}

\section{Conclusion and open problems}

On the search for APN permutations on even dimensions larger than $n=6$, Browning {\em et al.} remarked \cite{Dillon2} 
\begin{quote}
[T]he highly structured decomposition of the $\kappa$ code raise the hope that much of the structure, if not all, should generalize to higher dimensions. Does it?
\end{quote}

We have presented a new infinite family of simply defined APN functions which satisfies the important ``subspace property'' of the Kim function $\kappa$. Unfortunately, the family we presented does not seem to be equivalent to permutations. We suspect the reason for that is that our family of functions reduces to Gold functions (just like the Kim function) on the subfield $\f{2^m}$, which are {\bf not} permutations for our case since $m$ is even (unlike the Kim function, since $m = 3$ it does reduce to a permutation). Therefore the following is an interesting problem.

\begin{problem}
Find an infinite family of APN functions which includes the Kim function and which satisfies the subspace property. 
\end{problem}
We remark that the families B.4, B.5 and B.6 found by Bracken {\em et al.} \cite{B3,B4} includes the Kim function as a special case. But, no other member of these families directly satisfies the subspace property (it may be the case ---albeit unlikely--- that a family in Table \ref{APNmulti} includes a CCZ-equivalent function which satisfies the property, but according to some computer experiments this does not seem to be the case for small $n$).  


We do not know any categorical reasons why the existing families cannot be equivalent to permutations. A way to attack this problem is to show that the zeroes of Walsh spectrum cannot contain two trivially intersecting subspaces of dimension $n$. 

\begin{problem}
Show that the Gold functions (or any existing family) are not equivalent to permutations.
\end{problem}

Since the Walsh zeroes of the most of the known families are difficult to describe, this problem may be even more difficult to resolve for the families other than the Gold functions for which the description of Walsh zeroes is well-known. For this reason the following problem makes sense.

\begin{problem}
Describe the zeroes of the Walsh transform of known APN families. 
\end{problem}

All of these problems are posed to help solving the following difficult ``big APN problem.''
\begin{problem}
Are there APN permutations on $\f{2^{2m}}$ for $m > 3$?
\end{problem}

\section*{Acknowledgments}
The author would like to thank Gary McGuire for many helpful comments.

\bibliographystyle{acm}

\bigskip
\hrule
\bigskip

\end{document}